\documentclass[11pt, amsfonts]{amsart}

\usepackage{amsmath,amssymb,amsthm,color,enumerate,comment}
\usepackage[all]{xy}
\usepackage{tikz}
\usepackage{todonotes}
\usepackage{url} 
\usepackage{hyperref} 

\numberwithin{equation}{section}

\SelectTips{eu}{12}

\newtheorem{thm}{Theorem}[section]

\newtheorem{theorem}[thm]{Theorem}
\newtheorem{proposition}[thm]{Proposition}
\newtheorem{lemma}[thm]{Lemma}
\newtheorem{corollary}[thm]{Corollary}

\theoremstyle{definition}

\newtheorem{definition}[thm]{Definition}

\theoremstyle{remark}

\newtheorem{remark}[thm]{Remark}

\newcommand{\NN}{\mathbb{N}}
\newcommand{\ZZ}{\mathbb{Z}}

\newcommand{\Hom}{\mathrm{Hom}}

\newcommand{\cref}{}

\newcommand{\RR}{\mathbb{R}}

\newcommand{\NNN}{\mathcal{N}}
\newcommand{\PPP}{\mathcal{P}}

\newcommand{\tf}{\tilde{f}}
\newcommand{\tg}{\tilde{g}}
\renewcommand{\th}{\tilde{h}}

\newcommand{\tv}{\tilde{v}}

\newcommand{\tw}{\tilde{w}}

\newcommand{\tgamma}{\tilde{\gamma}}

\newcommand{\teta}{\tilde{\eta}}

\newcommand{\tx}{\tilde{x}}
\newcommand{\ty}{\tilde{y}}
\newcommand{\tz}{\tilde{z}}
\newcommand{\tG}{\tilde{G}}
\newcommand{\tH}{\tilde{H}}

\newcommand{\Homeo}{\mathrm{Homeo}}

\title[Hom complexes of graphs whose codomains are square-free]{Hom complexes of graphs whose codomains are square-free}

\author{Takahiro Matsushita}
\address{Department of Mathematical Sciences, Faculty of Science, Shinshu University, Matsumoto, Nagano 390-8621, Japan}
\email{matsushita@shinshu-u.ac.jp}

\subjclass[2020]{Primary 05C15; Secondary 55U10}

\keywords{Hom complex; square free graph; homomorphism reconfiguration}

\begin{document}

\begin{abstract}
The Hom complex $\Hom(G, H)$ of graphs is a simplicial complex associated to a pair of graphs $G$ and $H$, and its homotopy type is of interest in the graph coloring problem and the homomorphism reconfiguration problem. In this paper, we show that if $G$ is a connected graph and $H$ is a square-free connected graph, then every connected component of $\Hom(G, H)$ is homotopy equivalent to a point, a circle, $H$ or a connected double cover over $H$. We also obtain a certain relation between the fundamental group of $\Hom(G,H)$ and realizable walks studied in the homomorphism reconfiguration problem.
\end{abstract}

\maketitle

\section{Introduction}
\subsection{Background and the main result}
Lov\'asz's proof \cite{Lovasz} of the Kneser conjecture is widely known as one of the most significant achievements in the application of algebraic topology to combinatorics. He introduced the neighborhood complex $\NNN(G)$, which is a simplicial complex associated to a graph $G$, and showed that there is a profound relation between the chromatic number of $G$ and the homotopy type of $\NNN(G)$. The neighborhood complex is a central object of topological combinatorics, and has been extensively studied (see \cite{Csorba1, MZ, MatsushitaDCG, MW, Walker, WrochnaJCTB2, Zivaljevic} for example).

The Hom complex $\Hom(G,H)$ of graphs is a simplicial complex associated to a pair of graphs $G$ and $H$, which is a generalization of the neighborhood complex $\NNN(H)$ of $H$. As is the case of the neighborhood complex, the Hom complex has also been studied in the context of the graph coloring problem. For example, when $G$ belongs to a certain class of graphs, called homotopy test graphs, there is a close relation between the chromatic number of $H$ and the homotopy type of $\Hom(G,H)$. Lov\'asz's bound for the chromatic number \cite{Lovasz} implies that $K_2$ is a homotopy test graph. Babson and Kozlov \cite{BK1, BK2} showed that complete graphs and cycle graphs having at least three vertices are homotopy test graphs. In such a context of the graph coloring problem, the homotopy types of Hom complexes have been studied by many authors (see \cite{DS1, MatsushitaEJC1, MatsushitaEJC2, SchultzAdvances, SchultzJCTA} for example). Recently, Hom complexes are also of interest in the study of the homomorphism reconfiguration (see \cite{DochtermannEJC, DS2, WrochnaSIAM}).

The purpose of this paper is to investigate the homotopy type of $\Hom(G,H)$ whose codomain $H$ is square-free. Since our main theorem (Theorem~\ref{theorem main}) is a generalization of the case that $H$ is a cycle graph, we recall the previous work on this case. The homotopy types of Hom complexes between cycles were determined by \v{C}uki\'c and Kozlov \cite{CK}. Recently, Fujii et al. \cite{Fujii} studied the homotopy type of $\Hom(G,H)$ when $H$ is a cycle graph, and proved the following theorem:

\begin{theorem}[Fujii et al. \cite{Fujii}] \label{theorem cycle}
Let $G$ be a finite connected graph and $n$ an integer at least $3$. Then every component of $\Hom(G,C_n)$ is contractible or homotopy equivalent to a circle.
\end{theorem}

By the famous fold theorem (see \cite{KozlovPAMS} for example), one can show that every connected component of $\Hom(G,C_4)$ is contractible. Hence their contribution is to show the case $n \ne 4$. Our main theorem is a generalization of this case:

\begin{theorem} \label{theorem main}
Let $G$ be a finite connected graph and $H$ a finite square-free connected graph. Then every component of $\Hom(G,H)$ is homotopy equivalent to one of the following spaces: One point space, a circle, the graph $H$ or a connected double cover of $H$. Moreover, if $G$ is non-bipartite, then every component of $\Hom(G,H)$ is contractible or homotopy equivalent to a circle.
\end{theorem}

\begin{remark}
Soichiro Fujii, Kei Kimura and Yuta Nozaki \cite{FKN} independently proved the above theorem. Since their approach is different from that of this paper, we decided to announce them as separate papers.
\end{remark}

Theorem~\ref{theorem main} follows from Theorem~\ref{theorem precise 1} and Corollary~\ref{corollary precise 1} (for our conventions, see Subsection~\ref{subsection notation}). Note that the homotopy type of a connected double cover of a connected graph $H$ does not depend on the choice of the connected double cover since the homotopy type of a connected graph is determined by its Euler characteristic.
Furthermore, when $H$ is neither contractible nor homotopy equivalent to a circle, we determine the number of components homotopy equivalent to $H$ or its connected double cover (see Corollary~\ref{corollary precise 1} and Proposition~\ref{proposition number of components}).

Here we discuss the universality problem, which is a fundamental problem in topological combinatorics described as follows: Given a construction $K(G)$ of a simplicial complex from a graph $G$, determine the class of homotopy or topological types of spaces that can arise as $K(G)$ for some graph $G$. For instance, Csorba \cite{Csorba2} showed that for every finite free $\ZZ/2$-simplicial complex $X$, there exists a graph $G$ such that $\Hom(K_2, G) \simeq_{\ZZ / 2} X$. Furthermore, Dochtermann \cite{DochtermannC} proved that if a connected graph $T$ contains at least one edge, then for any finite simplicial complex $X$, there exists a (reflexive) graph $G$ such that $X \simeq \Hom(T,G)$.

On the other hand, Theorem~\ref{theorem main} implies that if $H$ is square-free, the class of spaces that can appear as $\Hom(G, H)$ is highly restricted. For instance, when $H$ is square-free, there does not exist a graph $G$ such that $\Hom(G, H) \simeq S^n$ for any $n \ge 2$.

\subsection{Construction of the universal covering} \label{subuniv}

The proof of Theorem~\ref{theorem main} is divided into three parts. First, we construct the universal covering of a connected component of $\Hom(G,H)$. Secondly, we show that the universal covering is contractible. Finally, we determine the fundamental group of each component of $\Hom(G,H)$. Such a rough outline of the proof is common to \cite{Fujii}.

When we consider the universal covering of a space of mappings between topological spaces $X$ and $Y$, it can sometimes be described as a space of equivariant maps between their universal coverings. For example, it is usual to describe the universal covering of the space $\Homeo^+(S^1)$ of orientation preserving homeomorphisms of a circle $S^1$ as follows:
\[ \widetilde{\Homeo^+}(S^1) = \{ f \in \Homeo(\RR) \; | \; \textrm{$f(x+1) = f(x) + 1$ for every $x \in \RR$}\}.\]
Our construction of the universal covering of a connected component of $\Hom(G,H)$ is similar to the geometric intuition above. However, we do not employ the usual covering map theory of graphs, since the usual covering maps are not so compatible with the topology of Hom complexes. Instead, we use the theory of $2$-covering maps, studied in \cite{MatsushitaJMSUT, TW, WrochnaJCTB1}.

Here, we briefly describe our construction of the covering of a component of $\Hom(G,H)$. In \cite{MatsushitaJMSUT}, the author introduced the $2$-fundamental group $\pi_1^2(G)$ of graphs and $2$-covering maps (see Section~\ref{section preliminaries}). Then, as is the case of the covering space theory (see \cite[Section~1.3]{Hatcher}) in classical topology, there are the $2$-coverings $p_G \colon \tG \to G$ and $p_H \colon \tH \to H$ which are \emph{universal}, that is, $\pi_1^2(\tG)$ and $\pi_1^2(\tH)$ are trivial. Let $f \colon G \to H$ be a graph homomorphism. Then the group $\Gamma = \pi_1^2(G)$ acts on $\tG$ and $\tH$, and we have a $\Gamma$-equivariant lift $\tf \colon \tG \to \tH$ of $f$.
We consider the $\Gamma$-equivariant Hom complex $\Hom^\Gamma(\tG, \tH)$, which consists of $\Gamma$-equivariant multi-homomorphisms. There is a natural surjective map $\pi \colon \Hom^\Gamma(\tG, \tH) \to \Hom(G,H)$, which turns out to be  a covering map (see Proposition~\ref{proposition Hom covering}).

Our main task of the proof is to show that if $H$ is square-free then this map $\pi$ gives a contractible covering of each component of $\Hom(G,H)$. Let $\Hom(G, H)_f$ denote the component of $\Hom(G,H)$ containing $f$, and $\Hom^\Gamma(\tG, \tH)_{\tf}$ the component of $\Hom^\Gamma(\tG, \tH)$ containing $\tf$. The following theorem is deduced from Theorem~\ref{theorem contractible 1}.

\begin{theorem} \label{theorem contractibility}
If $H$ is square-free, then $\Hom^\Gamma(\tG,\tH)_{\tf}$ is contractible.
\end{theorem}

\subsection{Fundamental groups and realizable walks}

The homotopy type of a connected simplicial complex $X$ with a contractible universal covering is determined by its fundamental group $\pi_1(X)$ (see \cite[Theorem 1B.8]{Hatcher}). Therefore, when $H$ is square-free, Theorem~\ref{theorem contractibility} implies that the homotopy type of $\Hom(G, H)_f$ is determined by its fundamental group. Thus, our next task is to study the fundamental group of $\Hom(G, H)_f$.

For the moment, we do not assume that \( H \) is square-free. Then, the covering \( \Hom^\Gamma(\tG, \tH)_{\tf} \) is not necessarilly simply connected (see Remark~\ref{remark not simply connected}). Nonetheless, we can describe a remarkable connection in the general situation between this covering \( \Hom^\Gamma(\tG, \tH)_{\tf} \) and the group \( \Pi(f, v) \) of realizable elements in \( \pi_1^2(H, f(v)) \), which is a key concept in Wrochna's work \cite{WrochnaSIAM} on the homomorphism reconfiguration problem.

We briefly recall his work. While Wrochna originally considered only the case where \( H \) is square-free, it is straightforward to extend his concept to the general setting in terms of \( 2 \)-fundamental groups. Let $v$ be a non-isolated vertex of $G$. Consider a \(\times\)-homotopy \( h \colon G \times I_n \to H \) from \( f \) to \( f \) (see Subsection~\ref{subsection Hom complex}). Then there is a closed walk 
\[ (h(v, 0), w_1, h(v, 1), w_2, \ldots, w_n, h(v, n)) \] 
at $f(v)$ with even length in \( H \). Although this closed walk depends on the choice of $w_k$, its \( 2 \)-homotopy class does not. Thus, it defines an element of \( \pi_1^2(H, f(v)) \). We call an element of \( \pi_1^2(H, f(v)) \) obtained in this manner \emph{realizable}. The realizable elements form a subgroup of \( \pi_1^2(H, f(v)) \), which we denote by \( \Pi(f, v) \).

We are now ready to state our next result, which reveals a relation among $\pi_1(\Hom^\Gamma(\tG, \tH)_{\tf})$, $\pi_1(\Hom(G,H)_f)$ and $\Pi(f,v)$.

\begin{theorem} \label{theorem exact}
Let $G$ and $H$ be connected graphs. Then there is a following short exact sequence of groups:
\[ 1 \to \pi_1(\Hom^\Gamma(\tG, \tH)_{\tf}, \tf) \to \pi_1(\Hom(G,H)_f, f) \to \Pi(f,v) \to 1.\]
\end{theorem}

When $H$ is square-free, Theorem~\ref{theorem contractibility} implies that $\Hom(G,H)_{f}$ is the $K(\Pi(f,v), 1)$-space. Wrochna \cite{WrochnaSIAM} determined the isomorphism types of $\Pi(f,v)$ in this case (see Theorem~\ref{theorem Wrochna}), which completes the proof of Theorem~\ref{theorem main}.

The rest of this paper is organized as follows: In Section~\ref{section preliminaries}, we review necessary definitions and facts. In Section~\ref{section construction} we construct the covering $\Hom^\Gamma(\tG, \tH)$ and show that the natural map $\pi \colon \Hom^\Gamma(\tG, \tH) \to \Hom(G,H)$ is a covering map (Proposition~\ref{proposition Hom covering}). In Section~\ref{section contractibility}, we show Theorem~\ref{theorem contractibility}. In Section~\ref{section fundamental group}, we prove Theorem~\ref{theorem exact}, and complete the proof of Theorem~\ref{theorem main}. In Appendix, we provide an alternative proof of Wrochna's result (Theorem~\ref{theorem Wrochna}) required for the proof of Theorem~\ref{theorem main}.

\section{Preliminaries} \label{section preliminaries}
In this section, we review necessary definitions and facts, following \cite{Hatcher, HN, Kozlovbook, MatsushitaJMSUT}.

\subsection{Graphs} \label{subsection graphs}
A \emph{graph} is a pair $G = (V(G), E(G))$ consisting of a set $V(G)$ together with a symmetric subset $E(G)$ of $V(G) \times V(G)$, i.e., $(x,y) \in E(G)$ implies $(y,x) \in E(G)$. Hence our graphs may have loops but no multiple edges. A vertex $v$ is \emph{looped} if $(v,v) \in E(G)$. If $G$ has no looped vertices, then $G$ is said to be \emph{simple}. For $x,y \in V(G)$, we write $x \sim_G y$ or $x \sim y$ to mean that $x$ and $y$ are adjacent.

For $n \ge 3$, the $n$-cycle graph $C_n$ is defined by $V(C_n) = \ZZ / n$ and $E(C_n) = \{ (x,x\pm 1) \; | \; x \in \ZZ / n\}$. A \emph{square-free graph} is a simple graph which has no subgraph isomorphic to $C_4$.

For graphs $G$ and $H$, a \emph{graph homomorphism from $G$ to $H$} is a map $f \colon V(G) \to V(H)$ such that $(v,w) \in E(G)$ implies $(f(v), f(w)) \in E(H)$.

Let $P_n$ be the path graph with $n + 1$ vertices. Namely, $V(P_n) = \{ 0,1,\ldots, n\}$ and $E(P_n) = \{ (x,y) \mid |x - y| = 1\}$. A \emph{walk in a graph $G$ with length $n$} is a graph homomorphism $\gamma \colon P_n \to G$, and set $\mathrm{length}(\gamma) = n$. A walk $\gamma$ is called a \emph{path} when $\gamma$ is injective as a set map. We sometimes write $(\gamma(0), \cdots, \gamma(n))$ to indicate the walk $\gamma$. For a pair of walks $\gamma \colon P_n \to G$ and $\delta \colon P_m \to G$ such that $\gamma(n) = \delta(0)$, the \emph{concatenation $\gamma \cdot \delta$ of $\gamma$ and $\delta$} is defined by the walk $(\gamma(0), \cdots ,\gamma(n), \delta(1), \cdots, \delta(m))$ with length $m + n$.

\begin{remark} \label{remark order of concatenation}
The order of concatenation of walks aligns with that in \cite{Hatcher, TW, WrochnaJCTB1} but is the reverse of the order used in \cite{MatsushitaJMSUT}. While this does not entail any essential change, it reverses the roles of the starting and terminal points of walks. Furthermore, whereas the action of the $2$-fundamental group on the universal $2$-covering is from the right in \cite{MatsushitaJMSUT}, this action is from the left in this paper (see Subsection~\ref{subsection 2-covering action}).
\end{remark}

Let $\gamma$ be a walk $(x_0, \ldots, x_n)$, and suppose $x \sim x_0$ and $x_n \sim y$. We denote the walk $(x, x_0, \ldots, x_n)$ by $x$-$\gamma$, the walk $(x_0, \ldots, x_n, y)$ by $\gamma$-$y$, and the walk $(x_0, \ldots, x_{n-1})$ by $\gamma \setminus x_n$.

A \emph{based graph} is a pair $(G,v)$ consisting of a graph $G$ equipped with a vertex $v$ of $G$. For based graphs $(G,v)$ and $(H,w)$, a \emph{based graph homomorphism from $(G,v)$ to $(H,w)$} is a graph homomorphism from $G$ to $H$ sending $v$ to $w$. A \emph{closed walk of a based graph $(G,v)$} is a walk from $v$ to $v$.

The \emph{categorical product $G \times H$ of graphs $G$ and $H$} is defined by
\[ V(G \times H) = V(G) \times V(H), \quad E(G \times H) = \{ ((x,y),(x',y')) \; | \; (x,x') \in E(G), (y, y') \in E(H)\}.\]

\subsection{Poset topology} \label{subsection poset}

Let $P$ and $Q$ be (possibly infinite) posets. A \emph{chain in $P$} is a subset $\sigma$ of $P$ such that any two elements of $\sigma$ are comparable. The \emph{order complex $\Delta(P)$ of $P$} is the simplicial complex whose simplices are finite chains in $P$. The \emph{classifying space of $P$} is the geometric realization of $\Delta(P)$, and denoted by $|P|$. We apply topological terms to posets, using classifying spaces. For example, a subposet $Q$ of $P$ is said to be a deformation retract of $P$ if $|Q|$ is a deformation retract of $|P|$.

A map $f \colon P \to Q$ is said to be \emph{monotone} or \emph{order preserving} if for any two elements $x,y$ of $P$, $x \le y$ implies $f(x) \le f(y)$. Then a monotone map induces a continuous map $|f| \colon |P| \to |Q|$ between their classifying spaces.

An \emph{ascending (or descending) closure operator of $P$} is a monotone map $c \colon P \to P$ such that $c(x) \ge x$ ($c(x) \le x$, respectively) for every $x \in P$ and $c^2 = c$.

\begin{theorem}[{\cite[Theorem 13.12]{Kozlovbook}}] \label{theorem closure operator}
Let $c \colon P \to P$ be an ascending or descending closure operator. Then $c(P)$ is a deformation retract of $P$.
\end{theorem}

\begin{definition} \label{definition poset covering map}
A monotone map $f \colon P \to Q$ is a \emph{poset covering map} if for every pair $y$ and $y'$ of elements of $Q$ with $y \le y'$, both of the following two conditions are satisfied:
\begin{enumerate}[(1)]
\item For every $x \in P$ with $f(x) = y$, there is unique $x' \in P$ such that $x \le x'$ and $f(x') = y'$.

\item For every $x' \in P$ with $f(x') = y'$, there is unique $x \in P$ such that $x \le x'$ and $f(x) = y$.
\end{enumerate}
\end{definition}

\begin{proposition}[{see \cite[Proposition 2.9 and Theorem 3.2]{BM} for example}] \label{proposition poset covering maps}
A poset covering map $f \colon P \to Q$ induces a covering map $|f| \colon |P| \to |Q|$ between their classifying spaces.
\end{proposition}

\subsection{Hom complexes} \label{subsection Hom complex}

For a set $X$, let $\PPP(X)$ denote the power set of $X$. For graphs $G$ and $H$, a \emph{multi-homomorphism from $G$ to $H$} is a map $\eta \colon V(G) \to \PPP(V(H)) \setminus \{ \emptyset \}$ such that $(x,y) \in E(G)$ implies $\eta(x) \times \eta(y) \subset E(H)$, i.e., every element of $\eta(x)$ and every element of $\eta(y)$ are adjacent. Let $\Hom(G,H)$ be the set of multi-homomorphisms from $G$ to $H$. We define a partial order on $\Hom(G,H)$ as follows: For $\eta, \eta' \in \Hom(G,H)$, $\eta \le \eta'$ if and only if $\eta(x) \subset \eta'(x)$ for every $x \in V(G)$. A map $f \colon V(G) \to V(H)$ is a graph homomorphism if and only if the map $x \mapsto \{ f(x)\}$ is a multi-homomorphism. From this point of view, we regard a graph homomorphism as an element of $\Hom(G,H)$.

A graph homomorphism $p \colon H_1 \to H_2$ induces a monotone map $p_* \colon \Hom(G,H_1) \to \Hom(G,H_2)$ defined by $p_*(\eta)(x) = p(\eta(x))$.

\begin{remark} \label{remark isolated}
If $G$ has no edge, then any map from $V(G)$ to $\PPP(V(H))$ is a multi-homomorphism. Thus, $\Hom(G,H)$ is empty or contractible, and Theorem~\ref{theorem main} is trivial. This is an exceptional case, and we assume that $G$ and $H$ have at least one edge in the subsequent sections (see Subsection~\ref{subsection notation}).
\end{remark}

Now we recall a bit of $\times$-homotopy theory, studied by Dochtermann \cite{DochtermannEJC}. For a non-negative integer $n$, let $I_n$ be the graph defined by $V(I_n) = \{ 0,1, \cdots, n\}$, $E(I_n) = \{ (i,j) \; | \; i,j \in V(I_n), |i-j| \le 1 \}$.
Let $f,g \colon G \to H$ be graph homomorphisms. A \emph{$\times$-homotopy from $f$ to $g$} is a graph homomorphism $h \colon G \times I_n \to H$ such that $h(x,0) = f(x)$ and $h(x,n) = g(x)$ for every $x \in V(G)$. Two graph homomorphisms $f$ and $g$ are said to be \emph{$\times$-homotopic} if there is a $\times$-homotopy from $f$ to $g$. The proof of the following lemma is straightforward.

\begin{lemma} \label{lemma x-homotopy}
For graph homomorphisms $f,g \colon G \to H$, the following are equivalent:
\begin{enumerate}[(1)]
\item The map $h \colon V(G \times I_1) \to V(H)$ defined by $h(x,0) = f(x)$ and $h(x,1) = g(x)$ is a graph homomorphism from $G \times I_1$ to $H$.

\item The map $f \cup g \colon V(G) \to \PPP(V(H)) \setminus \{ \emptyset\}$, $x \mapsto \{ f(x), g(x)\}$ is a multi-homomorphism.

\item There is a multi-homomorphism $\eta \in \Hom(G,H)$ such that $f \le \eta$ and $g \le \eta$.
\end{enumerate}
\end{lemma}

This implies the following:

\begin{proposition}[{\cite[Proposition 4.7]{DochtermannEJC}}] \label{proposition homotopy}
Let $f$ and $g$ be graph homomorphisms from $G$ to $H$. Then $f$ and $g$ are $\times$-homotopic if and only if $f$ and $g$ belong to the same connected component of $\Hom(G,H)$.
\end{proposition}

\subsection{$2$-fundamental groups} \label{subsection 2-fundamental group}

The theory of $2$-fundamental groups and $2$-covering maps was first studied in \cite{MatsushitaJMSUT}, and appllied to multiplicative graphs in \cite{TW, WrochnaJCTB1}. Another type of combinatorial fundamental groups related to Hom complexes was studied in \cite{DochtermannJCTA}.

Let $v,v' \in V(G)$, and let $\Omega (G,v,v')$ be the set of walks joining $v$ to $v'$. Let $\simeq_2$ be the equivalence relation on $\Omega(G,v,v')$ generated by the following two relations (A) and (B): For walks $(x_0, \cdots, x_m)$ and $(y_0, \cdots, y_n)$ joining $v$ to $v'$,
\begin{enumerate}[(A)]
\item $n = m + 2$ and there is $k \in \{ 0,1,\ldots, m\}$ such that $x_i = y_i$ for every $i \le k$ and $x_i = y_{i+2}$ for every $i \ge k$.

\item $m = n$ and $\# \{ i \mid x_i \ne y_i\} \le 1$.
\end{enumerate}
Namely, for $\gamma, \gamma' \in \Omega(G, v, v')$, $\gamma \simeq_2 \gamma'$ holds if and only if there is a sequence $\gamma = \gamma_0, \cdots, \gamma_k = \gamma'$ of $\Omega(G,v, v')$ such that for each $i = 1, \cdots, k$ one of $(\gamma_{i-1}, \gamma_i)$ and $(\gamma_i, \gamma_{i-1})$ satisfies one of (A) and (B).

Let $\pi_1^2(G,v,v')$ be the quotient set $\Omega(G,v,v') / {\simeq_2}$. We call the equivalence class containing $\varphi \in \Omega(G,v,v')$ the \emph{$2$-homotopy class of $\varphi$}. We write $\pi_1^2(G,v)$ instead of $\pi_1^2(G,v,v)$. The concatenation of walks defines a group operation on $\pi_1^2(G,v)$. We call this group the \emph{$2$-fundamental group of $(G,v)$}. Note that a based graph homomorphism $f \colon (G,v) \to (H,w)$ induces a group homomorphism $f_* \colon \pi_1^2(G,v) \to \pi_1^2(H,f(v))$.

We write $\Omega(G,v)$ instead of $\Omega(G,v,v)$. Let $\simeq_1$ be the equivalence relation on $\Omega(G,v)$ generated by the condition (A) and set $\pi_1^1(G,v) = \Omega (G,v)/ \simeq_1$. Similarly, by the concatenation of closed walks, $\pi_1^1(G,v)$ becomes a group. We call this group the $1$-fundamental group of $(G,v)$. By the definition there is a natural quotient map $\pi_1^1(G,v) \to \pi_1^2(G,v)$. When $G$ is simple, then the $1$-fundamental group is isomorphic to the edge path group (see \cite{Spanier}) of $(G,v)$, and hence is isomorphic to the usual fundamental group. Thus, in this case we write $\pi_1(G,v)$ instead of $\pi_1^1(G,v)$.

If $G$ is square-free, then  (B) is obtained from (A). Thus we have the following:

\begin{lemma} \label{lemma square-free isomorphism}
If $G$ is square-free and $v \in V(G)$, then the natural quotient map $\pi_1(G,v) \to \pi_1^2(G,v)$ is an isomorphism.
\end{lemma}

By the definition of the $2$-fundamental group $\pi_1^2(G,v)$, there is a well-defined group homomorphism
\[ \pi_1^2(G,v) \to \ZZ / 2, \quad [\gamma] \mapsto \mathrm{length}(\gamma)\mod 2.\]
The kernel of this group homomorphism is called the \emph{even part of $\pi_1^2(G,v)$} and is denoted by $\pi_1^2(G,v)_{ev}$. If the component of $G$ containing $v$ is bipartite, then $\pi_1^2(G,v)_{ev}$ coincides with $\pi_1^2(G,v)$. If the component of $G$ containing $v$ is non-bipartite, then $\pi_1^2(G,v)_{ev}$ is a subgroup of $\pi_1^2(G,v)$ whose index is $2$.

Finally, we recall the following close relationship between the $2$-fundamental groups and the fundamental groups of neighborhood complexes.

\begin{theorem}[Theorem 1.1 of \cite{MatsushitaJMSUT}] \label{theorem neighborhood complex}
Let $(G,v)$ be a based graph, and assume that $v$ is not isolated. Then there is a natural isomorphism
\[ \pi_1^2(G,v)_{ev} \xrightarrow{\cong} \pi_1(\NNN(G),v).\]
\end{theorem}

\subsection{$2$-covering maps} \label{subsection 2-covering}

\begin{definition}
Let $G$ be a graph and $x$ a vertex of $G$. Define $N_G(x)$ and $N_G^2(x)$ by
\[ N_G(x) = \{ y \in V(G) \mid (x,y) \in E(G) \}, \quad N_G^2(x) = \bigcup_{y \in N_G(x)} N_G(y).\]
We write $N(x)$ (or $N^2(x)$) instead of $N_G(x)$ (or $N^2_G(x)$, respectively).
\end{definition}

\begin{definition}
A \emph{covering map} or a \emph{$1$-covering map} is a graph homomorphism $p \colon G \to H$ such that $p|_{N(v)} \colon N(v) \to N(p(v))$ is bijective for every $v \in V(G)$. A \emph{$2$-covering map} is a graph homomorphism $p \colon G \to H$ such that for every $v \in V(G)$, the maps $p|_{N(v)} \colon N(v) \to N(p(v))$ and $p|_{N^2(v)} \colon N^2(v) \to N^2(p(v))$ are bijective.
\end{definition}

A $2$-covering map $p\colon G \to H$ is said to be \emph{connected} if $G$ is connected. A based graph homomorphism $p \colon (G,v) \to (H,w)$ is called a \emph{based $2$-covering map} if the underlying graph homomorphism $p \colon G \to H$ is a $2$-covering map.

\begin{theorem}[{\cite[Proposition 6.9]{MatsushitaJMSUT}}]
Let $(G, v)$ be a based graph. Then there is a connected based $2$-covering map $p \colon (\tilde{G}, \tilde{v}) \to (G,v)$ such that $\pi_1^2(\tilde{G},\tilde{v})$ is trivial. Moreover, such a based $2$-covering over $(G,v)$ is unique up to isomorphisms.
\end{theorem}

We call this graph $\tG$ the \emph{universal $2$-covering of $G$}.

\begin{remark} \label{remark universal bipartite}
The universal $2$-covering $\tG$ of $G$ is always simple and bipartite. Indeed, $\pi_1^2(\tG,\tv) \cong 1$ implies that there is no odd element of $\pi_1^2(\tG,\tv)$, and this means that $\tG$ has neither loops nor odd cycles.
\end{remark}

We will frequently use the following analogy to the lifting criterion of covering spaces \cite[Proposition 1.33]{Hatcher}.

\begin{lemma}[{lifting criterion \cite[Lemma 6.8]{MatsushitaJMSUT}}] \label{lemma lifting}
Let $p \colon (G,v) \to (H, w)$ be a based $2$-covering map, and $f \colon (T,x) \to (H, w)$ be a based graph homomorphism. Assume that $T$ is connected. Then the following are equivalent:
\begin{enumerate}
\item $f_* (\pi_1^2(T,x)) \subset p_*(\pi_1^2(G,v))$.

\item There is a based graph homomorphism $\tilde{f} \colon (T,x) \to (G,v)$ such that $p \circ \tilde{f} = f$.

\item There is a unique based graph homomorphism $\tilde{f} \colon (T,x) \to (G,v)$ such that $p \circ \tilde{f} = f$.
\end{enumerate}
\end{lemma}

\subsection{$2$-covering actions} \label{subsection 2-covering action}

\begin{definition}
Let $\Gamma$ be a group and $G$ a graph. A \emph{$\Gamma$-action on $G$} is a $\Gamma$-action on $V(G)$ such that for every $\gamma \in \Gamma$, the map $x \mapsto \gamma x$ is a graph homomorphism from $G$ to $G$.

For a $\Gamma$-action on a graph $G$, define the quotient graph $G/ \Gamma$ as follows: The vertex set of $G/ \Gamma$ is the orbit set $V(G)/ \Gamma$. The edge set of $G/ \Gamma$ is defined by
\[ E(G/ \Gamma) = \{ (\alpha, \beta) \in V(G / \Gamma) \times V(G / \Gamma) \mid (\alpha \times \beta) \cap E(G) \ne \emptyset \}.\]
In other words, $(\alpha, \beta) \in E(G/ \Gamma)$ if and only if there are $x \in \alpha$ and $y \in \beta$ satisfying $(x, y) \in E(G)$.
\end{definition}

It is easy to see that the projection $p_G \colon G \to G / \Gamma$ is a graph homomorphism.

\begin{definition}[the sentence before Proposition 6.10 of \cite{MatsushitaJMSUT}]
A $\Gamma$-action on a graph $G$ is a \emph{$2$-covering action} if $N^2(x) \cap N^2(\gamma x) = \emptyset$ for every $x \in V(G)$ and $\gamma \in \Gamma \setminus \{ 1_\Gamma\}$.
\end{definition}

Then the following equivalence holds:

\begin{proposition}[{see \cite[Proposition 6.10]{MatsushitaJMSUT}}] \label{proposition 2-covering action}
Let $G$ be a $\Gamma$-graph, and assume that the $\Gamma$-action is free as a set action. Then the following are equivalent:
\begin{enumerate}
\item The projection $G \to G/\Gamma$ is a $2$-covering map.

\item The $\Gamma$-action on $G$ is a $2$-covering action.
\end{enumerate}
\end{proposition}

Let $(G,v)$ be a connected based graph and $(\tG,\tv)$ the universal $2$-covering over $(G,v)$. Then $\pi_1^2(G,v)$ acts on $\tG$ as follows: Let $\tx \in V(\tG)$, and let $\varphi$ be a walk joining $\tv$ to $\tx$. Let $\alpha \in \pi_1^2(G,v)$ and $\gamma$ a representative of $\alpha$. Then there is a walk $\psi$ on $\tG$ such that $p_G \circ \psi = \gamma \cdot (p_G \circ \varphi)$ starting from $\tv$. Define $\alpha \tx$ to be the terminal point of $\psi$. This is the description of the $2$-covering action of $\pi_1^2(G,v)$ on $\tG$ (see the proof of \cite[Theorem 6.11]{MatsushitaJMSUT}). In \cite{MatsushitaJMSUT}, the author considered a right action, which is due to the fact that the order of concatenation of walks is the reverse of that in this paper (see Remark~\ref{remark order of concatenation}).

This description of the $\pi_1^2(G,v)$-action on $\tG$ implies the following:

\begin{lemma} \label{lemma description of action}
Let $p_G \colon (\tG, \tv) \to (G,v)$ be the universal $2$-covering of $(G,v)$, and let $g \in \pi_1^2(G,v)$. Let $\varphi$ be a walk of $\tG$ joining $\tv$ to $g \tv$. Then the $2$-homotopy class of the closed walk $p_G \circ \varphi$ of $(G,v)$ coincides with $g$.
\end{lemma}

Let $p_G \colon (\tG, \tv) \to (G,v)$ and $p_H \colon (\tH, \tw) \to (H,w)$ be the universal $2$-coverings. Let $f \colon (G,v) \to (H,w)$ be a based graph homomorphism. Then Lemma~\ref{lemma lifting} implies that there is a based graph homomorphism $\tf \colon (\tG,\tv) \to (\tH,\tw)$ satisfying $p_H \circ \tf = f \circ p_G$. Set $\Gamma = \pi_1^2(G,v)$. Then $\Gamma$ acts on $\tG$ as above. Moreover, $\Gamma$ acts on $\tH$ by $g \tw = f_*(g) \ty$ for $g \in \Gamma$ and $\ty \in V(\tH)$. By the above description of the actions on the universal $2$-coverings, the graph homomorphism $\tf \colon \tG \to \tH$ is $\Gamma$-equivariant.

\subsection{A bit of algebraic topology} \label{subsection covering space}

\begin{lemma}[{\cite[Example~2 and Theorem~A in Appendix]{Milnorbook}}] \label{lemma homotopy colimit}
Let $X$ be a simplicial complex and $(X_n)_{n \in \NN}$ an increasing sequence of subcomplexes of $X$. Suppose that $X_i$ is a deformation retract of $X_{i+1}$ for each $i$ and $X = \bigcup_{n \in \NN} X_n$. Then the inclusion $X_0 \hookrightarrow X$ is a homotopy equivalence.
\end{lemma}

\begin{theorem}[{see \cite[Proposition 1.39]{Hatcher} and its proof}] \label{theorem covering}
Let $(X,x_0)$ be a based simplicial complex and $p \colon (Y, y_0) \to (X,x_0)$ a connected covering space. Let $\Gamma_0$ be a group, and suppose that $\Gamma_0$ acts on $Y$ freely and $p(\gamma y) = p(y)$ for every $\gamma \in \Gamma_0$ and for every $y \in Y$. If
\[ p^{-1} (x_0) = \{ \gamma y_0 \; | \; \gamma \in \Gamma_0 \},\]
then there is a following short exact sequence of groups:
\[ 1 \to \pi_1(Y,y_0) \to \pi_1(X,x_0) \xrightarrow{\Psi} \Gamma_0 \to 1.\]
Moreover, the group homomorphism $\Psi \colon \pi_1(X,x_0) \to \Gamma_0$ is described as follows: Let $[\gamma] \in \pi_1(X,x_0)$. Then there is a unique $\tgamma \colon [0,1] \to Y$ such that $\tgamma(0) = y_0$ and $p \circ \tgamma = \gamma$. Then define $\Psi([\gamma])$ by $\tgamma(1) = \Psi([\gamma]) \cdot y_0$.
\end{theorem}

\subsection{Notation and conventions} \label{subsection notation}
For the rest of this paper, $G$ and $H$ are assumed to be finite non-empty connected graphs having at least one edge (see Remark~\ref{remark isolated}). The graph $H$ is not assumed to be square-free or even simple unless it is specified. Our results are valid for graphs having loops if it is not specified. Let $f \colon G \to H$ denote a graph homomorphism. Let $v \in V(G)$ and set $w = f(v)$. Let $\Gamma$ denote the $2$-fundamental group $\pi_1^2(G,v)$ of $(G,v)$. Let $p_G \colon (\tG, \tv) \to (G,v)$, $p_H \colon (\tH, \tw) \to (H,w)$ be the universal $2$-coverings of $(G,v)$ and $(H,w)$, respectively. Let $\tf \colon \tG \to \tH$ be a lift of $f$, and suppose $\tf(\tv) = \tw$.

\section{Construction of the covering} \label{section construction}

As was described in Subsection~\ref{subsection 2-covering action}, $\Gamma = \pi_1^2(G,v)$ acts on $\tG$ and $\tH$, and the lift $\tf \colon \tG \to \tH$ of $f$ is $\Gamma$-equivariant. A multi-homomorphism $\eta \in \Hom(\tG, \tH)$ is said to be \emph{$\Gamma$-equivariant} if $f_*(\gamma) \eta(\tx) = \eta(\gamma \tx)$ for every $\gamma \in \Gamma$ and $\tx \in V(\tG)$. Define $\Hom^\Gamma(\tG, \tH)$ to be the induced subposet of $\Hom(\tG, \tH)$ consisting of $\Gamma$-equivariant multi-homomorphisms.

Define the map
\[ \pi \colon \Hom^\Gamma(\tG, \tH) \to \Hom(G,H)\]
by $\pi(\eta)(p_G(\tx)) = p_H(\eta(\tx))$. It is straightforward to see that $\pi$ is a well-defined monotone map. The goal of this section is to prove that the map $\pi$ is a poset covering map.

\begin{proposition} \label{proposition Hom covering}
The map $\pi \colon \Hom^\Gamma(\tG, \tH) \to \Hom(G,H)$ is a poset covering map.
\end{proposition}
\begin{proof}
Let $\eta, \eta' \in \Hom(G, H)$ and suppose $\eta \le \eta'$. By the definition of poset covering maps (Definition~\ref{definition poset covering map}), it suffices to show the following:
\begin{enumerate}
\item For every $\teta \in \Hom^\Gamma(\tG,\tH)$ with $\pi(\teta) = \eta$, there is unique $\teta' \in \Hom^\Gamma(\tG,\tH)$ such that $\teta \le \teta'$ and $\pi(\teta') = \eta'$.
\item For every $\teta' \in \Hom^\Gamma(\tG,\tH)$ with $\pi(\teta') = \eta'$, there is unique $\teta \in \Hom^\Gamma(\tG,\tH)$ such that $\teta \le \teta'$ and $\pi(\teta) = \eta$.
\end{enumerate}

We first show (1). Choose $v_x \in \teta(x) \subset V(\tH)$ for each $x \in V(\tG)$ so that $v_{\gamma x} = \gamma v_x$. Then set $\teta'(x) = (p_H|_{N_{\tH}^2(v_x)})^{-1}(\eta'(x))$, where $p_H|_{N^2_{\tH}(v_x)} \colon N^2_{\tH}(v_x) \to N^2_H(p_H(v_x))$. Since $\tG$ has no isolated vertices (see Subsection~\ref{subsection notation}), we have $\teta(x) \subset N^2_{\tH}(v_x)$ and $p_H(\teta(x)) = \eta(x) \subset \eta'(x)$. Thus we have $\teta(x) \subset \teta'(x)$. Now we check the $\Gamma$-equivariance of $\teta'$. Since $\gamma v_x = v_{\gamma x}$, we have $\gamma \teta'(x) \subset \gamma N^2_{\tH}(v_x) \subset N^2_{\tH}(v_{\gamma x})$. Since $p_H(\gamma \teta'(x)) = \eta'(p_G(x)) = \eta'(p_G(\gamma x)) = p_H(\teta'(\gamma x))$, we have $\gamma \teta'(x) = \teta'(\gamma x)$. Hence $\teta'$ is $\Gamma$-equivariant.

We show that $\tilde{\eta}'$ is a multi-homomorphism. Let $(x,y) \in E(\tG)$, and let $a \in \teta'(x)$ and $b \in \teta'(y)$. It suffices to show that $a \sim_{\tH} b$. Since $p_H(a) \in \eta'(x) \subset N_H(p_H(v_y))$, there is $a' \in N_{\tH}(v_y)$ such that $p_H(a) = p_H(a')$. Since $v_x \sim_{\tH} v_y \sim_{\tH} a'$, we have $a' \in N^2_{\tH}(v_x)$. Since $a,a' \in N^2_{\tH}(v_x)$, we have $a = a'$ from the injectivity of $p_H|_{N^2_{\tH}(v_x)}$. Therefore we have $a = a' \sim_{\tH} v_y$. Similarly, we have $v_x \sim_{\tH} b$. Hence we have
\[ a \sim_{\tH} v_y \sim_{\tH} v_x \sim_{\tH} b.\]
Note that $p_H(a) \in \eta'(x)$, $p_H(b) \in \eta'(y)$ and $\eta'$ is a multi-homomorphism. Hence we have $p_H(a) \sim_{H} p_H(b)$. Then there is $b' \in N_{\tH}(a)$ such that $p_H(b') = p_H(b)$. Since $b,b' \in N_{\tH}^2(v_y)$, we have $b = b'$. Therefore we conclude that $a \sim b' = b$ and $\teta'$ is a multi-homomorphism.

Finally, we mention the uniqueness of $\teta'$. Suppose that $\teta'' \in \Hom^\Gamma(\tG, \tH)$ such that $\eta \le \teta''$ and $\pi(\teta'') = \eta'$. Since $\teta''$ is a multi-homomorphism, we have $\teta''(x) \subset N^2_{\tH}(v_x)$ for every $x \in V(\tG)$. Since $p_H |_{N^2_{\tH}(v_x)}$ is bijective and $p_H(\teta''(x)) = \eta'(p_G(x)) = p_H(\teta'(x))$, we have $\teta'(x) = \teta''(x)$. This completes the proof of (1).

The proof of (2) is easier. Indeed, for every $x \in V(\tG)$, let $v_x \in \teta'(x)$ so that $\gamma v_x = v_{\gamma x}$. Define $\teta(x)$ by $\teta(x) = (p_H|_{N^2(v_x)})^{-1}(\eta(x))$. Then $\teta(x)$ is clearly a multi-homomorphism since $\teta(x) \subset \teta'(x)$ and $\teta'$ is a multi-homomorphism. The proofs for verifying the other conditions are the same as that of (1). This completes the proof.
\end{proof}

\begin{remark} \label{remark Hom covering}
In the above proof, the hypothesis used in most parts is that $p_H \colon \tH \to H$ is a $2$-covering map. Hence, other variations of this proposition, such as the following, can also be shown by the same argument: If $p \colon H_1 \to H_2$ be a $2$-covering map, then the map $p_* \colon \Hom(G,H_1) \to \Hom(G, H_2)$ is a poset covering map. In fact, this version appeared in an earlier draft of \cite{MatsushitaJMSUT} (see \cite[Lemma 7.1]{MatsushitaJMSUT2}).
\end{remark}

\section{Contractibility} \label{section contractibility}

 As is the usual case, every $\Gamma$-equivariant graph homomorphism from $\tG$ to $\tH$ is regarded as an element of $\Hom^\Gamma(\tG, \tH)$. Hence we can regard $\tf \in \Hom^\Gamma(\tG, \tH)$, and let $\Hom^\Gamma(\tG, \tH)_{\tf}$ denote the component of $\Hom^\Gamma(\tG,\tH)$ containing $\tf$.

It is straightforward to see that if $H$ is square-free then $\tH$ is a tree (see \cite{TW} for example). Thus, to prove Theorem~\ref{theorem contractibility}, it suffices to show the following theorem.

\begin{theorem} \label{theorem contractible 1}
If $\tH$ is a tree, then $\Hom^\Gamma(\tG,\tH)_{\tf}$ is contractible.
\end{theorem}

The goal of this section is to prove Theorem~\ref{theorem contractible 1}. Throughout this section, we assume that \emph{$\tH$ is a tree}.

\subsection{Outline of the proof} \label{subsection outline}
Here we outline the proof of Theorem~\ref{theorem contractible 1}. We construct the filtration of $\Hom^\Gamma(\tG, \tH)_{\tf}$ in three stages as follows:
\begin{itemize}
\item For each $n \ge 0$, we construct a subposet $X_n$ such that
\[ \{ \tf\} = X_0 \subset X_1 \subset X_2 \cdots, \quad \Hom^\Gamma(\tG, \tH)_{\tf} = \bigcup_{n \ge 0} X_n.\]

\item Fix $n \ge 0$. For each $m \ge -1$, we construct a subposet $X_{n,m}$ such that
\[ X_{n-1} = X_{n,-1}, \quad X_{n,-1} \subset X_{n,0} \subset X_{n,1} \subset \cdots,\]
and $X_{n,m} = X_n$ for sufficiently large $m$.

\item Fix $n \ge 0$ and $m \ge -1$. For each $i \ge 0$, we construct a subposet $X_{n,m,i}$ such that
\[ X_{n,m,0} = X_{n,m-1}, \quad X_{n,m,0} \subset X_{n,m,1} \subset X_{n,m,2} \subset \cdots,\]
and $X_{n,m,i} = X_{n,m}$ for sufficiently large $i$.
\end{itemize}

Finally, we show that $X_{n,m,i-1}$ is a deformation retract of $X_{n,m,i}$ (Proposition~\ref{proposition deformation retract}), which concludes the proof of Theorem~\ref{theorem contractible 1} by Lemma~\ref{lemma homotopy colimit}.

\subsection{Subposet $X_n$}

In this subsection we define the subposet $X_n$ of $\Hom^\Gamma(\tG, \tH)$ for each $n \ge 0$. For $\eta \in \Hom^\Gamma(\tilde{G}, \tilde{H})_{\tilde{f}}$ and for $\tx \in V(\tilde{G})$, set
\[ u_\eta(\tx) = \max \{ d_{\tH}(\tf(\tx), \ty) \; | \; \ty \in \eta(\tx)\}.\]
Here $d_{\tH}$ denotes the path metric of $\tH$.

\begin{lemma}
For  every $\eta \in \Hom(\tG, \tH)$ and for every $\tx \in V(\tG)$, the integer $u_\eta(\tx)$ is even.
\end{lemma}
\begin{proof}
Since $\tH$ is bipartite (see Remark~\ref{remark universal bipartite}), there is a graph homomorphism $\varepsilon \colon \tH \to K_2$. Then $\varepsilon$ induces a map $\varepsilon_* \colon \Hom^\Gamma(\tG, \tH)_{\tf} \to \Hom(\tG, K_2) \cong S^0$. Since $u_\eta(\tx)$ is odd, $\eta(\tx)$ has an element $\ty$ such that $d_{\tH}(\tf(\tx), \ty)$ is odd. Thus, $\varepsilon_*(\eta)(\tx)$ contains an element different from $\varepsilon_* \tf(\tx) = \varepsilon(\tf(\tx))$. Hence we have $\varepsilon_*(\eta) \ne \tf$. This is a contradiction since $\Hom(\tG, K_2) \cong S^0$ and $\Hom^\Gamma(\tG,\tH)_{\tf}$ is connected.
\end{proof}

Since $\eta$ is $\Gamma$-equivariant, for every $\gamma \in \Gamma$ and $\tx \in V(\tG)$, we have $u_\eta(\gamma\tx) = u_\eta(\tx)$. Hence $u_\eta$ induces a function $\bar{u}_\eta$ on $V(G)$. Since $G$ is finite, $\bar{u}_\eta$ is bounded, which implies that $u_\eta$ is also bounded. Define the subposet $X_n$ of $\Hom^\Gamma(\tG, \tH)_{\tf}$ by
\[ X_n = \{ \eta \in \Hom^\Gamma(\tG, \tH)_{\tf} \; | \; \textrm{$u_\eta(\tx) \le 2n$ for every $\tx \in V(\tG)$}\}\]
Since $u_\eta$ is bounded, $\eta \in X_n$ holds for some $n$. Thus we have
\[ \{ \tf \} = X_0 \subset X_1 \subset X_2 \subset \cdots,  \quad \Hom^\Gamma(\tG, \tH)_{\tf} = \bigcup_{n \ge 0} X_n.\]

\subsection{Directed subgraphs $\tG(\eta)$ and $G(\eta)$}

Our next task is to construct the filtration $X_{n,m}$ of $X_n$. To define the filtration, we need to introduce a directed subgraph $\tG(\eta)$ of $\tG$ and a directed subgraph $G(\eta)$ of $G$ for each $\eta \in X_n$.

First we define $\tG(\eta)$ and $G(\eta)$ as undirected graphs. Let $\eta \in X_n$. Define the subgraph $\tG(\eta)$ of $\tG$ to be the induced subgraph of $\tG$ whose vertex set is $u_\eta^{-1}(2n)$. Similarly, define the subgraph $G(\eta)$ of $G$ to be the induced subgraph of $G$ whose vertex set is $\bar{u}_\eta^{-1}(2n)$. Next, we consider the direction of each edge of $\tG(\eta)$ and $G(\eta)$. Here we recall that $\tilde{H}$ is a tree, and that for any two vertices of $\tH$ there is a unique path joining them.

\begin{proposition} \label{proposition direction}
Let $(\tx,\tx') \in \tG(\eta)$. Then one of the following two conditions holds, and the other does not.
\begin{enumerate}
\item The set $\eta(\tx)$ consists of one vertex $\ty$ of $\tH$ and for every element $\ty' \in \eta(\tx')$ with $d_{\tH}(\tf(\tx'), \ty') = 2n$, the path joining $\tf(\tx')$ to $\ty'$ cointains $\ty$.

\item The set $\eta(\tx')$ consists of one vertex $\ty'$ of $\tH$ and for every element $\ty \in \eta(\tx)$ with $d_{\tH}(\tf(\tx), \ty) = 2n$, the path joining $\tf(\tx)$ to $\ty$ cointains $\ty'$.
\end{enumerate}
\end{proposition}
\begin{proof}
Since $\tx' \in V(\tG(\eta))$, we have $u_\eta(\tx') = 2n$. Hence there is $\ty' \in \eta(\tx')$ such that $d_{\tH}(\tf(\tx'),\ty') = 2n$. Similarly, there is $\ty \in \eta(\tx)$ such that $d_{\tH}(\tf(\tx), \ty) = 2n$. Let $P$ be the path in $\tH$ joining $\tf(\tx)$ to $\ty$.

Suppose that $\ty'$ is not contained in $P$. Then the path $P$ can be extended to the path $P$-$\ty'$. Since $\tH$ is a tree, we have $d_{\tH}(\tf(\tx), \ty') = 2n+1$. Let $\ty_0 \in \eta(\tx)$. Since $\eta \in X_n$, $d_{\tH}(\tf(\tx), \ty_0) \le 2n$ and $\ty_0$ is adjacent to $\ty'$. Since $\tH$ is a tree, this implies that $\ty_0$ coincides with $\ty$. If $\tf(\tx')$ is not contained in $P$, then $\tf(\tx')$-$P$-$\ty'$ is a path joining $\tf(\ty)$ to $\ty'$ with length $2n+2$. This means that $d_{\tH}(\tf(\tx'), \ty') = 2n+2$, and this is a contradiction. Hence $\tf(\tx')$ is contained in the path $P$, and is the next point from the starting point $\tf(\tx)$. Hence the path joining $\tf(\tx')$ to $\ty'$ contains $\ty$, and this means that (1) holds but (2) does not hold.

On the other hand, suppose that $\ty'$ is contained in $P$. Since $\tH$ is a tree, $\ty'$ is the previous point of the terminal point $\ty$ of $P$. Since $d_{\tH}(\tf(\tx'), \ty') = 2n$, $\tf(\tx')$ is not contained in $P$. Then $Q = \tf(\tx')$-$P \setminus \ty$ is the path joining $\tf(\tx')$ to $\ty'$. Thus, applying the above paragraph to $Q$ instead of $P$, we conclude that (2) holds but (1) does not.
\end{proof}

We define the direction on an edge of $\tG(\eta)$ as follows: If $(\tx,\tx') \in E(\tG(\eta))$ satisfies (1) of Proposition~\ref{proposition direction}, then we consider the source of $e$ is $\tx$ and the target of $e$ is $\tx'$, and if the other holds, then we consider the source of $e$ is $\tx'$ and the target of $e$ is $\tx$. Set
\[ \vec{E}(\tG(\eta)) = \{ (\tx, \tx') \in E(\tG(\eta)) \; | \; \textrm{the source of the edge is $\tx$ and the target of the edge is $\tx'$}\}\]
This direction on $\tG(\eta)$ is $\Gamma$-equivariant, i.e., for every edge $(\tx,\tx') \in \vec{E}(\tG(\eta))$, $(\gamma \tx, \gamma \ty) \in \vec{E}(\tG(\eta))$. Therefore, the direction on $\tG(\eta)$ induces the direction on $G(\eta)$. We define the set $\vec{E}(G(\eta))$ in the same way.

\begin{lemma} \label{lemma weak}
Let $e = (\tx, \tx') \in E(\tG(\eta))$. Then the following are equivalent:
\begin{enumerate}
\item $(\tx, \tx') \in \vec{E}(\tG(\eta))$.

\item There are $\ty \in \eta(\tx)$ and $\ty' \in \eta(\tx')$ such that $d_{\tH}(\tf(\tx), \ty) = d_{\tH}(\tf(\tx'), \ty') = 2n$ and the path joining $\tf(\tx')$ to $\ty'$ contains $\ty$.
\end{enumerate}
\end{lemma}
\begin{proof}
By the definition of the direction of edges, $(1)$ clearly implies $(2)$. Next, suppose (2). It suffices to show that (2) of Proposition~\ref{proposition direction} does not hold. By (2) of this lemma, the path $P$ in $\tH$ joining $\tf(\tx')$ to $\ty'$ contains $\ty$. Then $\ty$ is a previous point of $\ty'$ in $P$. Since $d_{\tH}(\tf(\tx), \ty) = 2n$, $\tf(\tx)$ is not contained in $P$, and $Q = \tf(\tx)$-$P \setminus \ty'$ is the path joining $\tf(\tx)$ to $\ty$. Then $Q$ does not contain $\ty'$ and hence (2) in Proposition~\ref{proposition direction} does not hold. This completes the proof.
\end{proof}

\begin{definition}
A vertex $\tx$ of $\tG(\eta)$ is a \emph{sink} if there is no edge of $\tG(\eta)$ whose source is $\tx$. Similarly, a vertex $x$ of $G(\eta)$ is a \emph{sink} if there is no edge of $G(\eta)$ whose source is $x$.
\end{definition}

\begin{lemma} \label{lemma direction comparison}
Let $\eta, \eta' \in X_n$ such that $\eta \le \eta'$. Then the following hold:
\begin{enumerate}
\item The directed graph $\tG(\eta)$ is a directed subgraph of $\tG(\eta')$.

\item Suppose that $e \in \vec{E}(\tG(\eta'))$ and the target of $e$ is not a sink of $\eta'$. Then $e \in \vec{E}(\tG(\eta))$.
\end{enumerate}
\end{lemma}
\begin{proof}
We first show (1). Then we have $V(\tG(\eta)) = u_{\eta}^{-1}(2n) \subset u_{\eta'}^{-1}(2n) = V(\tG(\eta'))$. To complete the proof of (1), it suffices to show that $e \in \vec{E}(\tG(\eta))$ implies $\vec{E}(\tG(\eta))$, i.e., $e$ has the same direction in $\tG(\eta)$ and $\tG(\eta')$. This clearly follows from Lemma~\ref{lemma weak}.

Next we show (2). Let $e = (\tx,\tx') \in \vec{E}(\tG(\eta'))$. By the assumption, there is an edge of $\tG(\eta')$ whose source is $\tx'$. This means that $\eta'(\tx)$ and $\eta'(\tx')$ are one point sets (see Proposition~\ref{proposition direction}). Thus we have $\eta(\tx) = \eta'(\tx)$ and $\eta(\tx') = \eta'(\tx')$, and $e$ has the same direction as that in $\tG(\eta)$.
\end{proof}

\begin{lemma} \label{lemma no cycle}
The directed graph $G(\eta)$ has no directed cycle.
\end{lemma}
\begin{proof}
Suppose that $G(\eta)$ has a directed cycle $\gamma = (x_0, \ldots, x_m)$ with $x_0 = x_m$, and let $\tilde{\gamma} = (\tx_0, \ldots, \tx_m)$ be a lift of $\gamma$ on $\tG$. Then there is $g \in \Gamma$ such that $\tx_m = g \tx_0$. Then we have a directed walk
\begin{align} \label{align path}
g^{-1} \tx_{m-1} \to \tx_0 \to \tx_1 \to \cdots \to \tx_m \to g \tx_1
\end{align}
in $\tG$. Proposition~\ref{proposition direction} implies that $\eta(\tx_0), \cdots, \eta(\tx_m)$ are one point sets. Let $\eta' \in \Hom^\Gamma(\tG, \tH)$. If $\eta' \le \eta$, then $\eta'(\tx_i) = \eta(\tx_i)$ for every $i$. If $\eta \le \eta'$, then Lemma~\ref{lemma direction comparison} implies that $\tG(\eta')$ contains the above directed path \eqref{align path}. Hence $\eta'(\tx_0), \ldots, \eta'(\tx_m)$ are also one point sets and hence we have $\eta(\tx_i) = \eta'(\tx_i)$. This means that for every $\eta'' \in \Hom^\Gamma(\tG, \tH)_{\tf}$, $\eta''(\tx_i) = \eta(\tx_i)$. This is a contradiction since $\tf \in \Hom^\Gamma(\tG, \tH)_{\tf}$ and $d_{\tH}(\tf(\tx_i), \eta(\tx_i)) = 2n > 0$.
\end{proof}

Since $G$ is a finite graph, Lemma~\ref{lemma no cycle} immediately deduces the following corollary:

\begin{corollary} \label{corollary sink}
The directed graph $G(\eta)$ has a sink, and so does $\tG(\eta)$.
\end{corollary}

\begin{lemma} \label{lemma not target}
Let $e = ( \tx, \tx' )\in E(\tG)$ and suppose $u_\eta(\tx) = 2n$. Let $\ty \in \eta(\tx)$ such that $d_{\tH}(\tf(\tx), \ty) = 2n$. Suppose that $\tx' \not\in V(\tG(\eta))$ or $(\tx', \tx) \in \vec{E}(\tG(\eta))$.
Then $\eta(\tx')$ consists of one point $\ty'$ and $\ty'$ is contained in the path $P$ joining $\tf(\tx)$ to $\ty$. Moreover, $\ty'$ is the previous point of the terminal point $\ty$ of $P$.
\end{lemma}
\begin{proof}
Let $\ty' \in \eta(\tx')$. Suppose that $\ty'$ is not contained in $P$. Then $P$-$\ty'$ is the path with length $2n+1$ joining $\tf(\tx)$ to $\ty'$. Since $\tf(\tx')$ is adjacent to $\tf(\tx)$ and $\eta \in X_n$, we have $d_{\tH}(\tf(\tx'), \ty') = 2n$ and $\tx' \in V(\tG(\eta))$. Since $P$ does not contain $\ty'$, we have $(\tx, \tx') \in \vec{E}(\tG(\eta))$. This contradicts the hypothesis.

Hence $\ty'$ is contained in $P$. Since $\ty$ and $\ty'$ are adjacent, $\ty'$ is the previous point of $\ty$. Since such a point is unique, the set $\eta(\tx')$ consists of one point $\ty'$.
\end{proof}

\subsection{Subposets $X_{n,m}$ and $X_{n,m,i}$}

We are now ready to introduce the filtrations $X_{n,m}$ and $X_{n,m,i}$ of $X_n$. Let $m$ be an integer at least $-1$. Define the induced subposet $X_{n,m}$ of $X_n$ as follows: $\eta \in X_{n,m}$ if and only if $G(\eta)$ has no directed path having the length greater than $m$. Here we consider that a vertex is a directed path with length $0$. Then we have $X_{n-1} = X_{n,-1}$ and $X_{n,m-1} \subset X_{n,m}$ for $m \ge 0$. Since $G$ is a finite graph, we have $X_n = X_{n,m}$ for sufficiently large $m$.

Next we define the subposet $X_{n,m,i}$. Let $P_1, \ldots, P_l$ be the directed paths with length $m$ in $G$. For $i \ge 0$, let $X_{n,m,i}$ be the induced subposet of $X_{n,m}$ consisting of multi-homomorphisms $\eta \in X_{n,m}$ such that $G(\eta)$ does not contain $P_{i+1}, \ldots, P_l$. Then $X_{n,m,0} = X_{n,m-1}$ and $X_{n,m,i} = X_{n,m}$ for every $i \ge l$. These are the desired conditions described in Subsection~\ref{subsection outline}.

\subsection{Deformation retracts}

Thus, to complete the proof of Theorem~\ref{theorem contractible 1}, it suffices to show the following (see Lemma~\ref{lemma homotopy colimit}):

\begin{proposition} \label{proposition deformation retract}
For $i \ge 1$, the poset $X_{n,m,i-1}$ is a deformation retract of $X_{n,m,i}$
\end{proposition}

The rest of this section is devoted to the proof of Proposition~\ref{proposition deformation retract}. We define an ascending closure operator $U \colon X_{n,m,i} \to X_{n,m,i}$ and a descending closure operator $D \colon U(X_{n,m,i}) \to U(X_{n,m,i})$ such that $DU(X_{n,m,i}) = X_{n,m,i-1}$.

First we define the map $U \colon X_{n,m,i} \to X_{n,m,i}$ as follows: Let $\eta \in X_{n,m,i}$. If $\eta \in X_{n,m,i-1}$, then we set $U(\eta) = \eta$. Suppose $\eta \not\in X_{n,m,i-1}$. Then $G(\eta)$ contains $P_i$.

Let $x$ be the terminal point of $P_i$. Let $\tx$ be a vertex of $V(\tG(\eta))$ such that $p_G(\tx) = x$. Then there is $\ty \in \eta(\tx)$ such that $d_{\tH}(\tf(\tx), \ty) = 2n$. Let $P$ be the path joining $\tf(\tx)$ to $\ty$. Let $\tx_{\eta}$ be the vertex on $P$ such that $d_{\tH}(\tf(\tx), \tx_{\eta}) = 2n-2$. Then $\tx_\eta$ does not depend on the choice of $\ty$. Indeed, let $\tx'$ be a vertex adjacent to $\tx$. Lemma~\ref{lemma not target} implies that $\eta(\tx')$ consists of one vertex $\ty'$, and $\ty'$ is the previous vertex of the terminal point $\ty$ of $P$. Then $\tx_\eta$ is the previous vertex of the terminal point of the path joining $\tf(\tx)$ to $\ty'$. Since $\eta(\tx') = \{ \ty'\}$ does not depend on the choice of $\ty \in \eta(\tx)$, $\tx_{\eta}$ does not depend on the choice of $\ty$.

Now we define the map $U(\eta) \colon V(\tG) \to \PPP(V(\tG)) \setminus \{ \emptyset \}$ as follows:
\[ U(\eta)(\tz) = \begin{cases}
\eta(\tz) \cup \{ \tz_\eta \} & (p_G(\tz) = x) \\
\eta(\tz) & (p_G(\tz) \ne x).
\end{cases}\]
Then $U(\eta)$ is clearly $\Gamma$-equivariant and $U^2 = U$. We would like to show the following:
\begin{enumerate}
\item $U(\eta)$ is a multi-homomorphism.

\item $\tG(\eta) = \tG(U(\eta))$. Similarly, $G(\eta) = G(U(\eta))$. In particular, $U(\eta) \in X_{n,m,i}$.

\item $U$ is monotone.
\end{enumerate}

First we show that $U(\eta)$ is a multi-homomorphism. Since $U(\eta)$ and $\eta$ coincide except for $p_G^{-1}(x)$, it suffices to observe edges connecting with a vertex contained in $p_G^{-1}(x)$. Let $\tx \in p_G^{-1}(x)$ and $\tx'$ a vertex adjacent to $\tx$. Let $\ty' \in \eta(\tx')$. It suffices to check that $\tx_{\eta}$ and $\ty'$ are adjacent. Let $\ty \in \eta(\tx)$ satisfying $d_{\tH}(\tf(\tx), \ty) = 2n$, and $P$ the path joining $\tf(\tx)$ to $\ty$. Then it follows from Lemma~\ref{lemma not target} that $\ty'$ is the previous point of the terminal point $\ty$ of $P$. Since $\tx_\eta$ is the vertex of $P$ with $d_{\tH}(\tf(\tx), \tx_\eta) = 2n-2$, $\tx_\eta$ and $\ty'$ are adjacent. This completes the proof of (1).

Next, we show $(2)$. It suffices to show only $\tG(\eta) = \tG(U(\eta))$. Since for every $\tz \in V(\tG)$, $U(\eta)(\tz)$ is just adding a vertex at distance $2n-2$ from $\tf(\tz)$, we have $V(\tG(\eta)) = V(\tG(U(\eta)))$ and the direction of each edge coincides (see Lemma~\ref{lemma weak}). This completes the proof of (2).

Finally, we show (3). Let $\eta, \eta' \in X_{n,m,i}$ and $\eta \le \eta'$. If $U(\eta) = \eta$, then $U(\eta) = \eta \le \eta' \le U(\eta')$. Hence we consider the case that $\eta < U(\eta)$. Then $\eta \not\in X_{n,m,i-1}$. Let $x$ be the terminal point of $P_i$ and $\tx \in p_G^{-1}(x)$. Since $U(\eta) \ne \eta$, $\eta(\tx)$ does not contain $\tx_\eta$. Since $\eta \le \eta'$, Lemma~\ref{lemma direction comparison} implies that $G(\eta')$ also contains $P_i$. If $x$ is not a sink in $G(\eta')$, then $G(\eta')$ contains a directed path whose length is greater than $m$. This contradicts the assumption that $\eta' \in X_{n,m}$. Hence $x$ is a sink of $G(\eta')$ and $\tx_{\eta'} = \tx_\eta$. Hence we have
\[ U(\eta')(\tilde{x}) = \eta'(\tilde{x}) \cup \{ \tilde{x}_{\eta'}\} \supset \eta(\tilde{x}) \cup \{ \tilde{x}_{\eta} \} = U(\eta)(x).\]
This completes the proof of (3).

Next we define the map $D \colon U(X_{n,m,i}) \to U(X_{n,m,i})$. If $\eta \in X_{n,m,i-1}$, then set $D(\eta) = \eta$. Suppose $\eta \not\in X_{n,m,i-1}$. Then $G(\eta)$ contains $P_i$. Let $x$ be the terminal point of $P_i$ and let $\tx \in p_G^{-1}(x)$. Since $\eta \in U(X_{n,m,i})$ and $U^2 = U$, $\eta(\tx)$ contains $\tx_{\eta}$. Hence we define $D(\eta) \in X_{n,m,i}$ by
\[ D(\eta)(\tz) = \begin{cases}
\{ \tz_\eta\} & (p_G(\tz) = x) \\
\eta(\tz) & (p_G(\tz) \ne x).
\end{cases}\]
Since $\emptyset \ne D(\eta)(\tz) \subset \eta(\tz)$ for each $\tz \in V(\tG)$ and $\eta$ is a multi-homomorphism, $D(\eta)$ is a multi-homomorphism. It is clear that $D(\eta)$ is $\Gamma$-equivariant. Note that $D(\eta)(\tx)$ does not contain a vertex $\ty$ such that $d_{\tH}(\tf(\tx), \ty) = 2n$. This means $\tx \not\in V(\tG(D(\eta)))$. By Lemma~\ref{lemma direction comparison}, we have $D(\eta) \in X_{n,m,i-1} \subset U(X_{n,m,i})$.

We show that $D \colon U(X_{n,m,i}) \to U(X_{n,m,i})$ is monotone. Let $\eta, \eta' \in U(X_{n,m,i})$. We only need to consider the case $D(\eta') < \eta'$. Let $x$ be the terminal point of $P_i$. Suppose that $D(\eta) \le D(\eta')$ does not hold. Then $\eta \le D(\eta')$ does not hold and $\eta \le \eta'$. Hence there is $\tx \in p_G^{-1}(x)$ and $\ty \in \eta(\tx) \subset \teta'(\tx)$ such that $d_{\tH}(\tf(\tx), \ty) = 2n$. By Lemma~\ref{lemma direction comparison}, $G(\eta)$ contains $P_i$. Since $\eta \in U(X_{n,m,i})$, we have $U(\eta) = \eta$, and hence we have that $\eta(\tx)$ contains $\tx_\eta$. This means that $\ty \not\in D(\eta)(\tx)$, this is a contradiction. This completes the proof that $D$ is monotone.

We have shown that $U \colon X_{n,m,i} \to X_{n,m,i}$ and $D \colon U(X_{n,m,i}) \to U(X_{n,m,i})$ are closure operators and $DU(X_{n,m,i}) = X_{n,m,i-1}$. By Theorem~\ref{theorem closure operator}, $X_{n,m,i-1}$ is a deformation retract of $X_{n,m,i}$. This completes the proof of Proposition~\ref{proposition deformation retract}.

\section{Fundamental group of $\Hom(G,H)_f$} \label{section fundamental group}

In this section, we study the fundamental group of $\Hom(G,H)_f$. In Subsection~\ref{subsection realizable}, we establish a relationship between the fundamental group of $\Hom(G,H)_f$ and realizable walks studied in Wrochna \cite{WrochnaSIAM} (Theorem~\ref{theorem exact}). In Subsection~\ref{subsection proof of main theorem}, we complete the proof of Theorem~\ref{theorem main}, using a fact (Theorem~\ref{theorem Wrochna}) proved in \cite{WrochnaSIAM}.

\subsection{Realizable elements} \label{subsection realizable}
In his study of the homomorphism reconfiguration problem, Wrochna \cite{WrochnaSIAM} considered the group $\Pi(f,v)$ for the case that $H$ is square-free. While his work deals only with the case where $H$ is square-free, it is straightforward to generalize it in terms of the 2-homotopy classes of walks (see also \cite{TW}).

For the moment, suppose that $H$ is an arbitrary connected graph. Let $h$ be a $\times$-homotopy from $f$ to $f$ (see Subsection~\ref{subsection Hom complex}), i.e., a graph homomorphism $h \colon G \times I_n \to H$ such that $h(x,0) = h(x,n) = f(x)$. Let $v' \in N_G(v)$. Then we have a closed walk
\[ f(v) = h(v,0) \sim h(v',0) \sim h(v,1) \sim \cdots \sim h(v',n-1) \sim h(v,n) = f(v)\]
of $H$ with even length. Although this closed walk depends on the choice of $v' \in N_G(v)$, its $2$-homotopy class $\alpha$ does not. In this case, we say that \emph{$\alpha$ is realized by $h$}. We call an element of $\pi_1^2(H,w)_{ev}$ \emph{realizable} if it is realized by a $\times$-homotopy. The realizable elements form a subgroup of $\pi_1^2(H, w)$, and we denote this group by $\Pi(f,v)$.

We now provide a characterization of the realizability in terms of $\Hom^\Gamma(\tG, \tH)$. In the following theorem, recall that $\pi_1^2(H,w)$ acts on the universal $2$-covering $\tH$ (Subsection~\ref{subsection 2-covering action}).

\begin{theorem} \label{theorem coincide}
Let $\varpi$ be an element of $\pi_1^2(H, w)$. Then the following are equivalent:
\begin{enumerate}[(1)]
\item $\varpi$ is realizable, i.e., $\varpi \in \Pi(f,v)$.

\item $\varpi \tf$ belongs to $\Hom^\Gamma(\tG, \tH)_{\tf}$. Here $\varpi \tf$ is a graph homomorphism sending $\tx$ to $\varpi (\tf(\tx))$.
\end{enumerate}
\end{theorem}
\begin{proof}
We first prove $(2) \Rightarrow (1)$. If (2) holds, by the same argument to show Proposition~\ref{proposition homotopy}, we have a $\times$-homotopy $\th \colon \tG \times I_n \to \tH$ from $\tf$ to $\varpi \cdot \tf$ such that for every $t \in \{ 0,1, \cdots, n\}$, the map $V(\tG) \to V(\tH)$, $\tx \mapsto \th(\tx,t)$ is $\Gamma$-equivariant. Let $\tv' \in N_{\tG}(\tv)$ and consider the walk
\[ \varphi = (\th(\tv,0), \th(\tv',0), \th(\tv,1), \th(\tv',1), \cdots \th(\tv, n))\]
in $\tG$. Then $\varphi$ joins $\tw$ to $\varpi \cdot \tw$. By Lemma~\ref{lemma description of action}, the $2$-homotopy class of $p_H \circ \varphi$ is $\varpi$.

On the other hand, $\th$ induces a $\times$-homotopy $h \colon G \times I_n \to H$ such that $h(p_G(\tx), t) = p_H(\th(\tx,t))$. Then the $2$-homotopy class realized by $h$ coincides with the $2$-homotopy class of $p_H \circ \varphi$, which is $\varpi$. This completes the proof of $(2) \Rightarrow (1)$.

Next we show $(1) \Rightarrow (2)$. Suppose that $\varpi$ is realizable. Let $h \colon G \times I_n \to H$ be a $\times$-homotopy from $f$ to $f$ such that the $2$-homotopy class realized by $h$ is $\varpi$. For $i = 0,1, \cdots, n$, define the graph homomorphism $f_i \colon G \to H$ by $f_i(x) = h(x,i)$. Then Lemma~\ref{lemma x-homotopy} implies that there is a sequence
\[ f = f_0 \le \eta_0 \ge f_1 \le \cdots \ge f_n = f\]
in $\Hom(G,H)$. By Proposition~\ref{proposition Hom covering}, there is a sequence
\[ \tf = \tf_0 \le \teta_0 \ge \tf_1 \le \cdots \ge \tf_n\]
in $\Hom^\Gamma (\tG, \tH)$ such that $\pi(\tf_i) = f_i$ and $\pi(\teta_i) = \eta_i$. Then we have a $\times$-homotopy $\th \colon \tG \times I_n \to \tH$ defined by $\th(\tx,i) = \tf_i(\tx)$. Note that the $2$-homotopy class of walks realized by $\th$ is a lift of that of $h$. Hence, by the description of the action on $\tH$ (see Subsection~\ref{subsection 2-covering action}), the $2$-homotopy class of walks realized by $\th$ joins $\tf(\tv)$ to $\varpi \tf(\tv)$. This means that $\tf_n(\tv) = \varpi \cdot \tf(\tv)$ and $p_H \circ \tf_n = p_H \circ (\varpi \tf)$. Thus, the uniqueness of the lifting criterion (Lemma~\ref{lemma lifting}) implies that $\varpi \tf = \tf_n \in \Hom^\Gamma(\tG, \tH)_{\tf}$.
\end{proof}

The following corollary is known for the case that $H$ is square-free (see \cite[Lemma~4.1]{WrochnaSIAM}).

\begin{corollary} \label{corollary commutative}
Let $\alpha \in \pi_1^2(G,v)$ and $\varpi \in \Pi(f,v)$. Then $f_*(\alpha) \varpi = \varpi f_*(\alpha)$ in $\pi_1^2(H,w)$.
\end{corollary}
\begin{proof}
It follows from Theorem~\ref{theorem coincide} that $\tf$ and $\varpi \tf$ are $\Gamma$-equivariant. Hence we have
\[ (f_*(\alpha) \varpi) \cdot \tf(\tv) = (\varpi \tf)(\alpha \tv) = (\varpi f_*(\alpha))(\tf(\tv)).\]
Since the action of $\pi_1^2(H, w)$ on $V(\tH)$ is free, we conclude $\varpi f_*(\alpha) = f_*(\alpha) \varpi$.
\end{proof}

\begin{lemma} \label{lemma covering}
Let $\tg \in \Hom^\Gamma(\tG, \tH)_{\tf}$ and suppose $\pi(\tg) = f$ (for the definition of $\pi \colon \Hom^\Gamma(\tG, \tH) \to \Hom(G,H)$, see Section~\ref{section construction}). Then there is $\varpi \in \pi^2_1(H, w)$ such that $\tg = \varpi \tf$.
\end{lemma}
\begin{proof}
Since $p_H(\tg(\tv)) = \pi(\tg)(v) = f(v) = p_H(\tf(\tv))$, there is $\varpi \in \pi_1^2(H,w)$ such that $\tg(\tv) = \varpi \tf(\tv)$. By the uniqueness of the lifting criterion (Lemma~\ref{lemma lifting}), we have $\tg = \varpi \tf$.
\end{proof}

We are now ready to prove Theorem~\ref{theorem exact}. For the reader's convenience, we restate the theorem below.

\medskip
\noindent
{\bf Theorem~\ref{theorem exact}.} 
\textit{There is a following short exact sequence of groups:}
\[ 1 \to \pi_1(\Hom^\Gamma(\tG, \tH)_{\tf}, \tf) \to \pi_1(\Hom(G,H)_f, f) \to \Pi(f,v) \to 1.\]

\begin{proof}[Proof of Theorem~\ref{theorem exact}]
This follows from Theorem~\ref{theorem covering}, Theorem~\ref{theorem coincide} and Lemma~\ref{lemma covering}.
\end{proof}

The following corollary is a direct consequence of Theorem~\ref{theorem exact}.

\begin{corollary} \label{corollary simply connected}
If $\Hom^\Gamma(\tG, \tH)_{\tf}$ is simply connected, then $\pi_1(\Hom(G,H)_f, f)$and $\Pi(f,v)$ are isomorphic.
\end{corollary}

\begin{remark} \label{remark not simply connected}
In general, $\Hom^\Gamma(\tG,\tH)_{\tf}$ is not necessarily simply connected. Indeed, consider the case that $G = K_n$, $H = K_{n+1}$ for $n \ge 3$, and $f \colon G \to H$ is any graph homomorphism. Since $\pi_1^2(H,f(v))_{ev} \cong \pi_1(\NNN(K_{n+1}), f(v)) \cong 1$ (Theorem~\ref{theorem neighborhood complex}). Since $\Pi(f, v) \subset \pi_1^2(H,f(v))_{ev}$, $\Pi(f,v)$ is trivial. Hence Theorem~\ref{theorem exact} implies that $\pi_1(\Hom^\Gamma(\tG, \tH)) \cong \pi_1(\Hom(G,H)_f)$. Since $\Hom(G,H)$ is a wedge of circles and is not contractible (see \cite[Section~4]{BK1}), $\pi_1(\Hom^\Gamma(\tG,\tH)_{\tf})$ is non-trivial in this case.
\end{remark}

Recall that if the universal covering of a simplicial complex $X$ is contractible, then the homotopy type of $X$ is determined by its fundamental group (see \cite[Theorem 1B.8]{Hatcher}). Thus for a group $\Pi$, we denote by $K(\Pi, 1)$ the simplicial complex such that its universal covering is contractible and $\pi_1(K(\Pi,1)) \cong \Pi$. Note that if $\Pi$ is trivial then $K(\Pi,1)$ is contractible, and that if $\Pi$ is an infinite cyclic group then $K(\Pi, 1)$ is a circle.

Theorem~\ref{theorem contractible 1} and Corollary~\ref{corollary simply connected} imply the following:

\begin{corollary} \label{corollary K}
If $\tH$ is a tree, then $\Hom(G,H)_f$ is $K(\Pi(f,v), 1)$.
\end{corollary}

We conclude this subsection by providing a direct description of the group homomorphism $\Psi \colon \pi_1(\Hom(G,H)_f, f) \to \Pi(f,v)$ in Theorem~\ref{theorem exact}. Let $\alpha \in \pi_1(\Hom(G,H)_f, f)$. Then there is an edge path (see \cite{Spanier})
\[ f = f_0 \le \eta_1 \ge \eta'_1 \le \eta_2 \ge \cdots \le \eta_n \ge f_n = f\]
in $\Hom(G,H)_f$ representing $\alpha$. For each $i = 1, \cdots, n-1$, let $f_i$ be a graph homomorphism such that $f_i \le \eta'_i$. Then the above edge path is homotopic to the edge path
\[ f = f_0 \le \eta_1 \ge f_1 \le \eta_2 \ge \cdots \le \eta_n \ge f_n.\]
Then Lemma~\ref{lemma x-homotopy} implies that we have a $\times$-homotopy $h \colon G \times I_n \to H$ from $f$ to $f$. The element realized by $h$ is $\Psi (\alpha)$.

To see this, by the same argument of the proof of Theorem~\ref{theorem coincide}, we can take
\[ \tf = \tf_0 \le \teta_1 \ge \tf_1 \le \teta_2 \ge \cdots \le \teta_n \ge \tf_n\]
a sequence of $\Hom^\Gamma(\tf,\tg)_{\tf}$ such that $\pi(\tf_i) = f_i$ and $\pi(\teta_i) = \eta_i$. By Theorem~\ref{theorem covering} we have $\tf_n = \Psi(\alpha) \cdot \tf$. Define the $\Gamma$-equivariant $\times$-homotopy $\th \colon \tG \times I_n \to \tH$ by $\th(\tx, i) = \tf_i(\tx)$. The $2$-homotopy class of walks realized by $\th$ joins $\tv$ to $\Psi(\alpha) \tv$, and $p_H \colon \tH \to H$ sends the $2$-homotopy class realized by $\th$ to the $2$-homotopy class realized by $h$. Thus Lemma~\ref{lemma description of action} implies that the $2$-homotopy class realized by $h$ is $\Psi(\alpha)$. This completes the proof.

\subsection{Proof of Theorem~\ref{theorem main}} \label{subsection proof of main theorem}

In this subsection, we complete the proof of Theorem~\ref{theorem main}.

When $H$ is square-free, Wrochna provided an algorithm to determine the isomorphism type of $\Pi(f,v)$ (see \cite[Theorem~8.1]{WrochnaSIAM}). However, to prove Theorem~\ref{theorem main}, it suffices to rely on a theorem of the following form.
Here, recall that $\pi_1(H,w)$ is a free group, and every subgroup of a free group is also a free group (see \cite[Proposition~1A.2 and Theorem~1A.4]{Hatcher}).

\begin{theorem}[Wrochna \cite{WrochnaSIAM}] \label{theorem Wrochna}
Assume that $H$ is square-free. Let $K$ be the image of the group homomorphism $f_* \colon \pi_1(G,v) \to \pi_1(H,w)$. Then the following hold:
\begin{enumerate}[(1)]
\item If $K$ is non-abelian free, then $\Pi(f,v)$ is trivial.

\item If $K$ is an infinite cyclic group, then $\Pi(f,v)$ is trivial or an infinite cyclic group.

\item If $K$ is trivial, then $\Pi(f,v)$ is $\pi_1(H,f(v))_{ev}$.
\end{enumerate}
\end{theorem}

In Appendix, we provide an alternative proof of Theorem~\ref{theorem Wrochna}. In this section, we complete the proof of the main theorem (Theorem~\ref{theorem main}) by Theorem~\ref{theorem Wrochna}. The following is the main part of Theorem~\ref{theorem main}.

\begin{theorem} \label{theorem precise 1}
Let $H$ a square-free graph. Let $K$ be the image of the group homomorphism
\[ f_* \colon \pi_1(G,v) \to \pi_1(H,w).\]
Then the following hold:
\begin{enumerate}[(1)]
\item If $K$ is a non-abelian free group, then $\Hom(G,H)_f$ is contractible.

\item If $K$ is an infinite cyclic group, then $\Hom(G,H)_f$ is contractible or homotopy equivalent to a circle.

\item If $K$ is trivial and $H$ is bipartite, then $\Hom(G,H)_f$ is homotopy equivalent to $H$.

\item If $K$ is trivial and $H$ is non-bipartite, then $\Hom(G,H)_f$ is homotopy equivalent to a connected double cover of $H$.
\end{enumerate}
\end{theorem}
\begin{proof}[Proof of Theorem~\ref{theorem precise 1}]
Suppose (1). Then Theorem~\ref{theorem Wrochna} implies that $\pi_1(\Hom(G,H)_f) \cong \Pi(f,v)$ is trivial, and hence $\Hom(G,H)_f = K(\Pi(f,v),1)$ (Corollary~\ref{corollary K}) is contractible. The proof of (2) is similar and is omitted. To see (3) and (4), suppose that $K$ is trivial. Then, by Theorem~\ref{theorem Wrochna} and Corollary~\ref{corollary K}, we have $\Hom(G,H)_f = K(\pi_1(H,v)_{ev},1)$. If $H$ is bipartite, we have $\pi_1(H,v)_{ev} = \pi_1(H,v)$, hence $\Hom(G,H)_f = K(\pi_1(H), 1) \simeq H$. If $H$ is non-bipartite, then $\pi_1(H,v)_{ev}$ is a subgroup of $\pi_1(H,v)$ index 2. Hence there is a connected double cover $H^*$ of $H$ such that $\pi_1(H^*)$ is isomorphic to $\pi_1(H,v)_{ev}$. This shows that $\Hom(G,H)_f = K(\pi_1(H^*), 1)$ is homotopy equivalent to $H^*$. This completes the proof.
\end{proof}

In particular, when $G$ is non-bipartite, we have the following corollary:

\begin{corollary} \label{corollary precise 1}
Let $G$ be a non-bipartite graph, and $H$ a square-free graph. Then every component of $\Hom(G,H)$ is contractible or homotopy equivalent to $S^1$.
\end{corollary}
\begin{proof}
Let $f \colon G \to H$ be a graph homomorphism. Since $G$ is non-bipartite, $\pi_1(G)$ has an odd element $\alpha$. Then $f_*(\alpha) \in \pi_1(H)$ is also an odd element, and hence is non-trivial. Therefore $f_* \colon \pi_1(G) \to \pi_1(H)$ is non-trivial, and Theorem~\ref{theorem precise 1} completes the proof.
\end{proof}

We are now ready to complete the proof of Theorem~\ref{theorem main}.

\begin{proof}[Proof of Theorem~\ref{theorem main}]
This follows from Theorem~\ref{theorem precise 1} and Corollary~\ref{corollary precise 1}.
\end{proof}

Next we consider the case where  $G$ is bipartite. Then there is a graph homomorphism $f \colon G \to H$ such that $f_* \colon \pi_1(G) \to \pi_1(H)$ is trivial.

\begin{proposition} \label{proposition number of components}
Let $G$ be a bipartite graph, and $H$ a square-free graph which is neither contractible nor homotopy equivalent to $S^1$. Then the following hold:
\begin{enumerate}[(1)]
\item If $H$ is bipartite, then the number of components of $\Hom(G,H)$ homotopy equivalent to $H$ is $2$.

\item If $H$ is non-bipartite, then the number of components of $\Hom(G,H)$ homotopy equivalent to a connected double cover of $H$ is $1$.
\end{enumerate}
\end{proposition}

Proposition~\ref{proposition number of components} follows from Lemma~\ref{lemma number of trivial} below and Theorem~\ref{theorem precise 1}. Although Lemma~\ref{lemma number of trivial} is likely a folklore fact, we provide the proof here for the reader's convenience.

\begin{lemma} \label{lemma number of trivial}
Let $G$ be a bipartite graph. Then the following hold:
\begin{enumerate}[(1)]
\item If $H$ is bipartite, then the number of components of $\Hom(G,H)$ containing a graph homomorphism $f$ which induces a trivial map $\pi_1(G) \to \pi_1(H)$ is $2$.

\item If $H$ is non-bipartite, then the number of components of $\Hom(G,H)$ containing a graph homomorphism $f$ which induces a trivial map $\pi_1(G) \to \pi_1(H)$ is $1$.
\end{enumerate}
\end{lemma}

The following lemma seems to be well-known, and easily follows from the fold theorem (see \cite[Theorem 3.3]{KozlovPAMS} for example) and Lemma~\ref{lemma homotopy colimit}.

\begin{lemma} \label{lemma tree tree}
Let $T$ be a (possibly infinite) tree having at least one edge. Then we have $\Hom(G,T) \simeq S^0$. Moreover, for two graph homomorphisms $f_0$, $f_1 \colon G \to T$, $f_0$ and $f_1$ belong to the same connected component of $\Hom(G,T)$ if and only if for some (or every) $x \in V(G)$, the distance $d_T(f_0(x), f_1(x))$ is even.
\end{lemma}

\begin{proof}[Proof of Lemma~\ref{lemma number of trivial}]
In this proof, we call a graph homomorphism $f$ \emph{null-homotopic} if the map $f_* \colon \pi_1(G) \to \pi_1(H)$ induced by $f$ is trivial.

Let $p \colon T \to H$ be the usual universal covering. Suppose that $f \colon G \to H$ is null-homotopic. By the usual lifting criterion \cite[Proposition~1.33]{Hatcher}, there is a graph homomorphism $f' \colon G \to T$ such that $p \circ f' = f$. Hence, $f$ is contained in the image of $p_* \colon \Hom(G, T) \to \Hom(G,H)$. Since $\Hom(G,T) \simeq S^0$ (Lemma~\ref{lemma tree tree}), the number of components of $\Hom(G,H)$ containing null-homotopic graph homomorphisms is at most $2$.

We now show (1). Suppose that $H$ is bipartite and let $\varepsilon \colon H \to K_2$ be a graph homomorphism. Then there are two null-homotopic graph homomorphisms $f_0, f_1 \colon G \to H$ such that $\varepsilon \circ f_0 \ne \varepsilon \circ f_1$. Since $\Hom(G, K_2) = S^0$, $f_0$ and $f_1$ belong to distinct components of $\Hom(G,H)$. Hence, the number of components containing null-homotopic graph homomorphisms is $2$. This completes the proof of (1).

Next we show (2). Suppose that $H$ is non-bipartite. Then there is an odd element $\alpha \in \pi_1(H)$. Then $f'$ and $\alpha f'$ belong to different components of $\Hom(G,T)$ since the length of path joining $f'(x)$ to $\alpha f'(x)$ is odd for every $x \in V(G)$ (see Lemma~\ref{lemma tree tree}). Since $p_H \circ f' = p_H\circ(\alpha f')$, $p_*$ sends two components of $\Hom(G,T)$ to the same component of $\Hom(G,H)$. Thus the number of components containing null-homotopic graph homomorphisms is $1$.
\end{proof}

\appendix

\section{Alternative proof of Theorem~\ref{theorem Wrochna}} \label{subsection alternative}

In Subsection~\ref{subsection proof of main theorem}, we proved Theorem~\ref{theorem main}, by using Theorem~\ref{theorem Wrochna}. In this appendix, we present an alternative proof of Theorem~\ref{theorem Wrochna} from the perspective of $2$-covering theory.

For a group $\Gamma_0$ and for a subset $S$ of $\Gamma_0$, let $Z_{\Gamma_0}(S)$ be the subgroup of $\Gamma_0$ consisting of elements of $\Gamma_0$ that commute with every element of $S$.

\begin{lemma} \label{lemma commute 1}
Let $F$ be a free group and $K$ a subgroup of $F$. Then the following hold:
\begin{enumerate}
\item If $K$ is a non-abelian free group, then $Z_F(K)$ is trivial.

\item If $K$ is an infinite cyclic group, then $Z_F(K)$ is an infinite cyclic group.
\end{enumerate}
\end{lemma}
\begin{proof}
We first show (1). Let $x \in Z_F(K)$ and consider the subgroup $K'$ of $F$ generated by $K \cup \{x\}$. Then, $K'$ is a non-abelian free group, and $x$ belongs to the center of $K'$. Since a non-abelian free group has the trivial center, we conclude that $x$ is trivial.

Next we show (2). Let $x \in K$ be a generator of $K$. Then $x \in Z_F(K)$ and hence $Z_F(K)$ is a non-trivial subgroup of $F$. If $Z_F(K)$ is not an infinite cyclic group, then $Z_F(K)$ is a non-abelian free group and $x \in Z_F(Z_F(K))$. Since $x$ is non-trivial, this contradicts (1).
\end{proof}

\begin{proof}[Alternative proof of Theorem~\ref{theorem Wrochna}]
Consider the following commutative diagram:
\[ \xymatrix{
\pi_1(G,v) \ar[r]^{f_*} \ar[d] & \pi_1(H,w) \ar[d]^{\cong} \\
\pi^2_1(G,v) \ar[r]^{f_*} & \pi_1^2(H,w).
}\]
The vertical arrows are the natural quotient maps (see Subsection~\ref{subsection 2-fundamental group}). Since $H$ is square-free, the right vertical arrow is an isomorphism (Lemma~\ref{lemma square-free isomorphism}). Hence $K = f_*(\pi_1(G,v))$ coincides with $f_*(\pi_1^2(G,v))$. Corollary~\ref{corollary commutative} implies $\Pi(f,v) \subset  Z_{\pi^2_1(H,w)}(K)$. Thus, (1) and (2) of Theorem~\ref{theorem Wrochna} follows from Lemma~\ref{lemma commute 1}.

Next we show (3). Since $K = f_*(\pi_1^2(G,v))$ is trivial, Lemma~\ref{lemma lifting} implies that there is a graph homomorphism $f' \colon G \to \tH$ such that $p_H \circ f' = f$. Since $H$ is square-free, $\tH$ is a tree. Let $\alpha \in \pi_1^2(H,w)_{ev}$. It follows from Lemma~\ref{lemma tree tree} that $f'$ and $\alpha f'$ belong to the same connected component of $\Hom(G, \tH)$. By Proposition~\ref{proposition homotopy}, there is a $\times$-homotopy $h \colon G \times I_n \to \tH$ from $f'$ to $\alpha f'$. Then $p_H \circ h$ is a $\times$-homotopy from $f$ to $f$. By Lemma~\ref{lemma description of action}, the $2$-homotopy class realized by $p_H \circ h$ is $\alpha$. Thus we have $\pi_1^2(H,w)_{ev} \subset \Pi(f,v)$. By the definition of $\Pi(f,v)$, it is clear that $\Pi(f,v) \subset \pi_1^2(H,w)_{ev}$. This completes the proof of (3).
\end{proof}

\section*{Acknowledgement}
The author was supported in part by JSPS KAKENHI Grant Numbers JP23K12975. The author thanks Soichiro Fujii, Kei Kimura and Yuta Nozaki for fruitful discussion. In particular, the discussions with them were important for coming up with the formulation of Theorem 1.2. The author also thanks Shun Wakatsuki for useful comments. The author is grateful to the anonymous referees for their helpful comments, which improve the readability of the manuscript.

\section*{Data availability}
No data was used for the research described in the article.

\section*{Declaration of generative AI and AI-assisted technologies in the manuscript preparation process}

During the preparation of this work the author used ChatGPT and Gemini in order to improve the readability of the manuscript. After using these tools, the author reviewed and edited the content as needed and takes full responsibility for the content of the published article.

\bibliographystyle{abbrvurl} %
\bibliography{reference} %

\end{document}